\documentclass[12pt,reqno]{amsart}

\usepackage{amsthm, amsmath, amsfonts, amssymb, graphicx, tikz-cd, calrsfs,hyperref} 

\newtheorem{theorem}{Theorem}[section]
\newtheorem{lemma}[theorem]{Lemma}
\newtheorem{proposition}[theorem]{Proposition}
\newtheorem{corollary}[theorem]{Corollary}

\theoremstyle{definition}
\newtheorem{definition}[theorem]{Definition}
\newtheorem{example}[theorem]{Example}

\newtheorem{remark}[theorem]{Remark}

\newcounter{mycount}

\newcommand{\myref}[1]{\hyperref[#1]{#1}}

\DeclareMathAlphabet{\mcal}{OMS}{zplm}{m}{n}
\def\F{\mathcal{F}}
\def\G{\mathcal{G}}
\def\C{\mathcal{C}}

\def\W{\mathcal{W}}

\def\sleq{\sqsubseteq}

\def\DL{{\sf DL}}

\def\up{{\uparrow}}
\def\down{{\downarrow}}

\begin{document}

\title{Remarks on hyperspaces for Priestley spaces}
\author{G.~Bezhanishvili, J.~Harding, P.~J.~Morandi}
\address{Department of Mathematical Sciences\\
New Mexico State University\\
Las Cruces, NM 88003\\
USA}
\email{guram@nmsu.edu, jharding@nmsu.edu, pamorand@nmsu.edu}


\begin{abstract}
The Vietoris space of a Stone space plays an important role in the coalgebraic approach to modal logic. When generalizing this to positive modal logic, there is a variety of relevant hyperspace constructions based on various topologies on a Priestley space and mechanisms to topologize the hyperspace of closed sets. A number of authors considered hyperspaces of Priestley spaces and their application to the coalgebraic approach to positive modal logic. A mixture of techniques from category theory, pointfree topology, and Priestley duality have been employed. Our aim is to provide a unifying approach to this area of research relying only on a basic familiarity with Priestley duality and related free constructions of distributive lattices.
\end{abstract}

\maketitle

\section{Introduction}

For a topological space $X$, a topology on the set of closed sets of $X$ is called a {\em hyperspace topology}. The classical examples are the Hausdorff distance and its generalization to the Vietoris space of a compact Hausdorff space (see e.g. \cite[Ch.~III-4]{Joh82}). Hyperspace constructions are a standard technique in topology \cite{Eng89} and play an important role in domain theory and theoretical computer science for modeling of nondeterministic choice and parallel processing \cite{GHKLMS03}. 
There are three classic constructions of powerdomains: the Smyth, Hoare, and Plotkin powerdomains (see, e.g., \cite[Ch.~IV-8]{GHKLMS03}). Smyth \cite{Smy83} pointed out a close connection of these three powerdomain constructions and hyperspace constructions in topology. On the other hand, Hennessy and Plotkin \cite{HP79} proved that powerdomain constructions are left adjoints to appropriate forgetful functors, and Abramsky \cite{Abr91} (see also \cite{AJ94}) provided connection to logic and Stone duality. 

The above three powerdomain constructions are closely related to hyperspace constructions for Priestley spaces, a type of ordered topological spaces that by Priestley duality \cite{Pri70,Pri72} are dual to bounded distributive lattices. There are several natural hyperspaces that one can associate with a Priestley space,
which are motivated by the coalgebraic approach to modal logic. Most standard of these is the so-called ``convex Vietoris space'' of a Priestley space, studied by a number of authors 
(see e.g.~\cite{Pal04,BK07,BKR07,VV14}). There are close links between this construction and the construction of the ``Vietoris frame" in pointfree topology \cite{Joh82,Joh85,Vic89}. Indeed, if $(X,\pi,\leq)$ is the Priestley dual of a bounded distributive lattice $L$, then it is shown in 
\cite{BK07,BKR07,VV14} that the convex Vietoris space of this Priestley space is the dual Priestley space of a bounded distributive lattice constructed from $L$ in a manner similar to the Vietoris frame construction in pointfree topology. Consequently, 
proofs in these papers rely on results from frame theory \cite{Joh85,Vic89}. 

The convex Vietoris space has also been studied from the perspective of bitopological spaces (see \cite{Lau15} and \cite[Ch.~4]{Jakl2017}). In particular, it is pointed out in \cite[Rem.~4.5.6]{Jakl2017} that a direct proof can be designed, avoiding the machinery of pointfree topology, by working with special pairs $(F,I)$ of filters and ideals of a distributive lattice that were also studied by Massas \cite{Massas2016}. Such pairs are a key ingredient in our treatment. Since they provide a generalization of prime filters and their complementary prime ideals, we call them weakly prime pairs (see below). 

Our purpose here is to provide a systematic self-contained account of various hyperspace constructions related to Priestley spaces. This includes the convex Vietoris space described above, as well as other hyperspaces arising from the hit, miss, and hit-or-miss topologies on the closed sets of each of the three natural topologies associated with a Priestley space: the Stone topology, the topology of open upsets, and the topology of open downsets. 

If $(X,\pi,\leq)$ is the Priestley space of $L$, we describe the corresponding hyperspaces concretely in terms of topologies on various sets of ideals and filters of $L$. Each of these hyperspaces is a Priestley space, or its close relative a spectral space, hence has a dual distributive lattice. These dual lattices are usually described freely, \`{a}~la Johnstone \cite{Joh85}. Here we also give their concrete descriptions in terms of $L$. The results obtained are used to give coalgebraic proofs of various dualities for positive modal logic. 

Many of these results exist in the literature in different forms (see, e.g., \cite{Pal04,BK07,BKR07,VV14,Lau15,Jakl2017}). 
Our approach provides a new perspective on the hyperspace constructions for Priestley spaces, gives direct descriptions of various free constructions, and fills some gaps in the literature. We also provide a unified and easily accessible treatment of the wide range of results related to hyperspaces of Priestley spaces.

The hit-or-miss topology \cite{Mat75} of a topological space $X$ is the topology on the set $\F$ of closed sets of $X$ having as subbasis all sets of the form 
\begin{align*}
\Diamond_U=&\, \{C\in\F\mid C\cap U\neq\varnothing\} \\
\Box_K=&\,\{C\in\F \mid C\cap K=\varnothing\} 
\end{align*}
where $U$ is open and $K$ is compact. When applied to a compact Hausdorff space $X$, where compact sets are the complements of opens, this yields the Vietoris space of $X$ (see e.g.~\cite[Ch.~III-4]{Joh82}). 
The hit-or-miss topology on $\F$ can be naturally viewed in three pieces.  The topology generated by the sets $\Diamond_U$ where $U$ is open in $X$ is the hit topology on $\F$; the topology generated by the $\Box_K$ where $K$ is compact in $X$ is the miss topology; and their join is the hit-or-miss topology. 

A Priestley space $(X,\pi,\leq)$ carries three natural topologies: the Stone topology $\pi$, and the spectral topologies $\pi^+$ and $\pi^-$ of open upsets and open downsets. The hit-or-miss topology applied to a locally compact $T_0$-space is a Priestley space under set inclusion \cite{Fel62,Hof85}. In particular, the hit-or-miss topology applied to the spectral spaces $\pi^+,\pi^-$ is Priestly. The hit and the miss topologies applied to $\pi^+,\pi^-$ are spectral. 

Suppose $(X,\pi,\leq)$ is the Priestley space of a bounded distributive lattice $L$. The hit-or-miss hyperspaces constructed from $\pi^+,\pi^-$ are Priestley spaces and are order-homeomorphic to ones constructed in a simple way from the ideals and from the filters of $L$. They are the Priestley spaces of lattices we call $L^\Box$ and $L^\Diamond$. The associated hit and miss topologies constructed from $\pi^+,\pi^-$ are the corresponding spectral spaces. The lattices $L^\Box$ and $L^\Diamond$ are realized via free constructions, modulo certain relations, over $L$ (see, e.g., \cite{BKR07}). We describe them concretely as sublattices of the powerset of $L$. 

One can also consider the hit-or-miss topology constructed from the Stone topology $\pi$. 
Here little new is gained since $\pi$ is the topology arising from the dual of the free Boolean extension $B_L$ of $L$. So this situation falls under those previously considered. A reason this case is less interesting is that the order of the Priestley space $(X,\pi,\leq)$ that was used to form the topologies $\pi^+,\pi^-$, and so is implicitly used in constructing the hit-or-miss topologies for these spaces, no longer plays a role in constructing the hit-or-miss topology on the hyperspace of $(X,\pi)$. 

A natural candidate to incorporate the order of the Priestley space in constructing the hit-or-miss topology on the hyperspace of $(X,\pi)$ is to use the Egli-Milner order $\sqsubseteq$ on the closed sets $\F$ (see e.g.~\cite{Smy91}). The Egli-Milner order is in general a quasi-order. There are two paths to obtain a partial ordering from it. 
The first is the usual approach to the Vietoris space of a Priestley space \cite{BK07,BKR07,VV14}, restricting $\sqsubseteq$ to the set $\C$ of convex closed sets where it is a partial order, then equipping $\C$ with an appropriate topology. The second \cite{Pal04} is to take the quotient $\F/{\sim}$ by the equivalence relation $\sim$ obtained from the quasi-order $\sqsubseteq$. With the quotient topology, $\F/{\sim}$ is a Priestley space. We show that the approaches of \cite{BK07,BKR07,VV14} and \cite{Pal04} are equivalent and that the topology used in the usual Vietoris construction \cite{BK07,BKR07,VV14} can be strictly weaker than the subspace topology. This answers a question raised in \cite{Pal04}. 

Let $L^{\Diamond\Box}$ be the bounded distributive lattice whose Priestley space is order-homeomorphic to the Vietoris space $\C$. In \cite{BK07,VV14} Johnstone's construction \cite{Joh85} of the Vietoris frame was used to show that $L^{\Diamond\Box}$ can be realized via a free construction over $L$. In \cite{BKR07} this result was obtained by a different approach, making use of \cite[Thm.~11.4.4]{Vic89}. Here we follow the path suggested in \cite[Rem.~4.5.6]{Jakl2017} and prove this result directly 
in terms of special pairs $(F,I)$ of a filter $F$ and an ideal $I$ of $L$, generalizing the pairs consisting of a prime filter and its complementary prime ideal. 
We also provide a concrete realization of $L^{\Diamond\Box}$ as a sublattice of the powerset of $B_L$ and describe it as an adjoint construction. 

Since Abramsky's pioneering work \cite{Abr88}, various connections between the hyperspace constructions in topology and duality theory in modal logic have been discovered. In particular, in  \cite{Abr88,KKV04} the Vietoris endofunctor on Stone spaces is used to provide a coalgebraic proof of J\'onsson-Tarski duality, which is so fundamental in modal logic. There are various generalizations of J\'onsson-Tarski duality to distributive lattices with operators \cite{Gol89,CLP91,Dun95,Pet96,Har98,CJ99}. Several authors \cite{Pal04,BK07,BKR07,VV14,Lau15,Jakl2017} have considered hyperspace endofunctors on Priestley and spectral spaces to provide coalgebraic proofs of these generalizations. We provide a brief account of these coalgebraic treatments from the unified perspective developed here.

As the first step, the various hyperspace constructions considered here together with their associated free constructions yield pairs of endofunctors on the categories $\sf PS$ of Priestley spaces and $\sf DL$ of bounded distributive lattices.
These pairs of endofunctors commute with the usual contravariant functors of Priestley duality, yielding that 
the associated categories of algebras and coalgebras for matching pairs of endofunctors are dually equivalent. 

Coalgebras for the endofunctor on $\sf PS$ that takes $X$ to the Priestley space of its convex closed sets $\C$ can be seen as Priestley spaces with a binary relation satisfying certain continuity conditions. The duality between these categories of algebras and coalgebras then yields a coalgebraic proof of Celani-Jansana duality \cite{CJ99} (see also \cite{Har98}).
Similar dualities are obtained using algebras and coalgebras for the endofunctors associated to the free constructions $L^\Diamond$ and $L^\Box$ on ${\sf DL}$ and the upper and lower Vietoris endofunctors on $\sf PS$. This gives a coalgebraic proof of Goldblatt duality \cite{Gol89} (see also \cite{CLP91,Pet96}). 

The paper is organized in the following way. The second section provides preliminaries. The third treats hyperspace topologies of spectral spaces. The fourth treats hyperspace topologies for Priestley spaces, both the essentially trivial version that ignores order, and the more interesting version involving the Egli-Milner order and convex sets. The fifth and final section gives coalgebraic treatment of dualities described above.

\section{Preliminaries} \label{sec: prelims}

\begin{definition}
For a topological space $X$ with closed sets $\F$, define for each open set $U$ and a compact set $K$ of $X$
\begin{align}
\Diamond_U=&\, \{C\in\F\mid C\cap U\neq\varnothing\} \label{b}\\
\Box_K=&\,\{C\in\F \mid C\cap K=\varnothing\} \label{a}
\end{align}
The topology generated by the sets $\Diamond_U$ where $U$ is open is the \emph{hit topology} on $\F$, the topology generated by the sets $\Box_K$ where $K$ is compact is the \emph{miss topology} on $\F$, and the join of these is the \emph{hit-or-miss topology} on $\F$. These hyperspaces are written $\F^\Diamond$, $\F^\Box$ and $\F^{\Diamond\Box}$. 
\end{definition}

Since $\Box_{K}\cap\Box_{L}=\Box_{K\cup L}$, it is easy to see that the sets $\Box_K$ where $K$ is compact form a basis for the miss topology. But we can do better. Recall (see e.g. \cite[p.~43]{GHKLMS03}) that a subset of a topological space is {\em saturated} if it is the intersection of open sets, and the {\em saturation} $A^s$ of a set $A$ is the intersection of the open sets that contain it. 

\begin{lemma} \label{lem:bases}
The hit topology of $\F$ has as a subbasis the sets $\Diamond_B$ where $B$ ranges over the members of a basis of the topology of $X$, and the miss topology of $\F$ has as a basis the sets $\Box_K$ where $K$ is compact and saturated. 
\end{lemma}

\begin{proof}
The first statement follows from the observation that if $U=\bigcup_I U_i$, then $\Diamond_U=\bigcup_I\Diamond_{U_i}$. The second statement follows from the observations that a closed set is disjoint from a set $A$ iff it is disjoint from its saturation $A^s$ and that the saturation of a compact set is compact. So $\Box_K$ is equal to $\Box_{K^s}$ and $K^s$ is compact and saturated. 
\end{proof} 

We will apply these notions in the setting of Priestley spaces and spectral spaces. To define these, we first recall that a subset of a topological space $X$ is {\em clopen} if it is both closed and open, that $X$ is {\em zero-dimensional} if clopen subsets of $X$ form a basis, and that $X$ is a {\em Stone space} if $X$ is compact, Hausdorff, and zero-dimensional.

\begin{definition}
For a poset $(X,\le)$ and $S\subseteq X$, let 
\begin{eqnarray*}
{\uparrow}S &=& \{x\in X \mid x\ge s \mbox{ for some } s\in S\} \\
{\downarrow}S &=& \{x\in X \mid x\le s \mbox{ for some } s\in S\}.
\end{eqnarray*} 
If $S=\{s\}$ we simply write ${\uparrow}s$ and ${\downarrow}s$. We call $S$ an {\em upset} if $S={\uparrow}S$ and a {\em downset} if $S={\downarrow}S$. 
\end{definition}

\begin{definition} 
A {\em Priestley space} is a triple $(X,\pi,\leq)$ where $(X,\pi)$ is a compact space and $\le$ is a partial order on $X$ that satisfies the following separation axiom: if $x\nleq y$, then there is a clopen upset $U$ containing $x$ and missing $y$.
\end{definition}

There is a close connection between Priestley spaces and distributive lattices. Let $\sf DL$ be the category of bounded distributive lattices and bounded lattice homomorphisms and let $\sf PS$ be the category of Priestley spaces and continuous order-preserving maps.

\begin{theorem}
There is a contravariant functor $P:{\sf DL}\to{\sf PS}$ with $P(L) = (X,\pi,\leq)$ where $X$ is the set of prime filters of $L$, $\leq$ is the partial ordering of set inclusion, and $\pi$ is the topology generated by sets $a^+,a^-$ for all $a\in L$ where 
\[
a^+=\{x\in X\mid a\in x\} \ \mbox{ and } \ a^-=\{x\in X\mid a\not\in x\}.
\]
For a morphism $f:L\to M$, $P(f)$ is the inverse image map $f^{-1}$. There is another contravariant functor $CU:{\sf PS}\to{\sf DL}$ taking a Priestley space to its lattice of clopen upsets and a morphism $g$ between Priestley spaces to its inverse image map $g^{-1}$ between their lattices of clopen upsets. Together, these contravariant functors provide a dual equivalence. 
\end{theorem}

For a Priestley space, there are three natural topologies of interest: the given topology $\pi$, the topology $\pi^+$ of open upsets, and the topology $\pi^-$ of open downsets. It is well known that $\pi=\pi^+\vee\pi^-$. The topologies $\pi$ arising as the topologies of a Priestley space are exactly Stone topologies. The topologies $\pi^+,\pi^-$ are characterized as spectral spaces, which are described as follows. 

\begin{definition}
A space is {\em sober} if every closed set that cannot be written as a union of two proper closed sets is the closure of a unique point, and it is {\em coherent} if the set of compact open sets is a basis that is a bounded sublattice of the lattice of all open sets. Spaces that are sober and coherent are called {\em spectral spaces}.
\end{definition}

\begin{definition}
For $L\in{\sf DL}$ with Priestley space $P(L)=(X,\pi,\leq)$, we let 
\[
P^+(L)=(X,\pi^+)\quad\mbox{ and }\quad P^-(L)=(X,\pi^-).
\]
\end{definition}

The space $P^+(L)$ has basis $\{a^+\mid a\in L\}$ and $P^-(L)$ has basis $\{a^-\mid a\in L\}$. Both spaces are spectral spaces. If $P(L)=(X,\pi,\leq)$, then using $L^d$ for the order-dual of $L$, we have $P(L^d)$ is order-homeomorphic to $(X,\pi,\geq)$, that $P^+(L^d)$ is homeomorphic to $P^-(L)$, and $P^-(L^d)$ is homeomorphic to $P^+(L)$. We don't retain the ordering from the Priestley space when passing to $P^+(L)$ since it does not play an essential role in our considerations, and it can be recovered from the topology $\pi^+$ as the specialization order. 

\begin{remark}
Historically spectral spaces preceded Priestley spaces. What we call $P^+(L)$ was originally produced by Stone \cite{Sto37b} by applying his famous construction to a bounded distributive lattice rather than a boolean algebra. For this reason $P^+(L)$ is often called the Stone space of $L$. 
\end{remark}

Throughout this note we assume that $(X,\pi,\leq)$ is a Priestley space, and use $\F, \F_+, \F_-$ for the sets of closed sets of the topologies $\pi$, $\pi^+$, and $\pi^-$ respectively. Thus, $\F$ is all closed sets, $\F_+$ is all closed downsets, and $\F_-$ is all closed upsets of $(X,\pi,\leq)$, each ordered by inclusion.

\begin{definition}
For a Priestley space $(X,\pi,\leq)$, we use $\F^\Diamond, \F^\Box, \F^{\Diamond\Box}$ for the hit, miss, and hit-or-miss topologies of $(X,\pi)$. Similarly, $\F_+^\Diamond, \F_+^\Box, \F_+^{\Diamond\Box}$ are these topologies for the topology $\pi^+$ of open upsets, and $\F_-^\Diamond, \F_-^\Box, \F_-^{\Diamond\Box}$ for the topology $\pi^-$ of open downsets.
\end{definition}

\begin{remark}
In \cite{Fel62} Fell generalized the Vietoris construction to produce a compact Hausdorff hyperspace for each locally compact space $X$. He took the closed sets $\F(X)$ and put the hit-or-miss topology on this hyperspace to produce what we call $\F(X)^{\Diamond\Box}$. Since each spectral space is locally compact, the hyperspace topologies $\F_+^{\Diamond\Box}$ and $\F_-^{\Diamond\Box}$ are the result of applying the Fell construction to the spectral spaces $(X,\pi^+)$ and $(X,\pi^-)$ respectively. 
\end{remark}

\begin{remark}
The terms upper and lower Vietoris space are often used in this context (see e.g.~\cite{Law11}). 
For example, for a Priestley space $(X,\pi,\leq)$ our $\F_-^\Box$ is called the upper Vietoris space of $(X,\pi^-)$ since its elements are closed upsets. Similarly, our $\F_+^\Diamond$ is called the lower Vietoris space of $(X,\pi^+)$ since its elements are closed downsets.  
\end{remark}

\section{The hyperspace topologies of a spectral space}\label{spectral}

Throughout this section $L\in {\sf DL}$ with Priestley space $(X,\pi,\leq)$. 
We consider hyperspace topologies on the closed sets $\F_+$ and $\F_-$ of the spectral spaces $P^+(L)=(X,\pi^+)$ and $P^-(L)=(X,\pi^-)$. 
We begin with focus on $\F_-$, the situation for $\F_+$ is similar. Our primary tool will be to realize $\F_-$ with its various topologies in terms of the filter lattice of $L$.

\begin{definition}
Let $\F(L)$ be the set of filters of $L$ ordered by inclusion. 
\end{definition}

For a filter $F$ of $L$ let $C_F$ be the closed upset $\bigcap\{a^+\mid a\in F\}$, and for a closed upset $C$ of $X$ let $F_C$ be the filter $\{a\in L\mid C\subseteq a^+\}$. Similarly, for an ideal $I$ of $L$ let $C_I$ be the closed downset $\bigcap\{a^-\mid a\in I\}$, and for a closed downset $C$ let $I_C$ be the ideal $\{a\in L\mid C\subseteq a^-\}$. The following is well-known (see e.g.~\cite[p.~54]{Pri84} or \cite[p.~385]{BBGK10}). 

\begin{proposition}\label{nmk}
There are mutually inverse dual order-isomorphisms $\Psi:\F_-\to\F(L)$ and $\Gamma:\F(L)\to\F_-$
taking a closed upset $C$ to the filter $\Psi(C)=F_C$ and a filter $F$ to the closed upset $\Gamma(F)=C_F$. Dually, there are order-isomorphisms between $\F_+$ and the ideal lattice $\mathcal I(L)$ taking a closed downset $C$ to the ideal $I_C$ and the ideal $I$ to the closed downset $C_I$.
\end{proposition}

Since $\F_-$ and $\F(L)$ are in bijective correspondence, any topology on $\F_-$ can be moved to one on $\F(L)$ so as to make the correspondence 
a homeomorphism. In particular, we can move the hit, miss, and hit-or-miss topologies to ones on $\F(L)$.

\begin{definition}
For $a\in L$ let $a^\uparrow=\{F\in\F(L)\mid a\in F\}$ and $a^\downarrow = \{F\in\F(L)\mid a\not\in F\}$. 
\end{definition}

Note that for $a=a_1\wedge\cdots\wedge a_n$ we have $a^\uparrow=a_1^\uparrow\cap\cdots\cap a_n^\uparrow$ and $a^\downarrow = a_1^\downarrow\cup\cdots\cup a_n^\downarrow$.

\begin{proposition}\label{prop: filter space}
The set $\F(L)$ of filters of $L$ is a Priestley space under the topology generated by $\{a^\uparrow,a^\downarrow \mid a\in L\}$ and the partial ordering of set inclusion of filters. The topology of open upsets of this Priestley space is generated by the basis $\{a^\uparrow \mid a\in L\}$ and that of open downsets is generated by the subbasis $\{a^\downarrow \mid a\in L\}$. 
\end{proposition}

\begin{proof}
Clearly the subbasis $\{a^\uparrow,a^\downarrow \mid a\in L\}$ consists of clopen sets since $a^\uparrow$ and $a^\downarrow$ are complements for each $a\in L$. If $F,G$ are filters and $F\not\subseteq G$, then there is $a\in F\setminus G$. So $F$ and $G$ are separated by the clopen sets $a^\uparrow$ and $a^\downarrow$. Suppose $\F(L)=\bigcup\{ a^\uparrow \mid a \in P \} \cup \bigcup\{ b^\downarrow \mid b \in Q\}$ is a cover by sets in the subbasis. Let $F$ be the filter generated by $Q$. Then $F\notin\bigcup\{ b^\downarrow \mid b\in Q\}$. Therefore, $F\in a^\uparrow$ for some $a\in P$. Then $a\in F$, so there are $b_1,\ldots,b_n\in Q$ with $b_1\wedge\cdots\wedge b_n\leq a$. This yields that $a^\uparrow,b_1^\downarrow,\dots,b_n^\downarrow$ is a finite subcover of $\mathcal{F}(L)$. It follows from Alexander's subbase lemma that $\mathcal{F}(L)$ is compact, and hence a Priestley space. 

Since $\mathcal{F}(L)$ is a Priestley space, the clopen upsets are a basis for the topology of open upsets. Suppose $U$ is a clopen upset and $F\in U$. For each $G\not\in U$ there is $a_{G}\in F\setminus G$. Then $\{a_{G}^\downarrow \mid G\not\in U\}$ covers the complement $U^c$. Since $U$ is clopen, $U^c$ is closed and hence compact since $\mathcal{F}(L)$ is compact. So there are $G_1,\ldots,G_n\not\in U$ with $U^c\subseteq a_{G_1}^\downarrow\cup\cdots\cup a_{G_n}^\downarrow$. Thus, $F\in a_{G_1}^\uparrow\cap\cdots\cap a_{G_n}^\uparrow\subseteq U$. Since each $a^\uparrow$ is an open upset, it follows that $\{a^\uparrow\mid a\in L\}$ is a subbasis for the topology of open upsets. A symmetric argument considering a clopen downset $V$ and $G\in V$ shows that $\{a^\downarrow\mid a\in L\}$ is a subbasis for the topology of open downsets. Finally, since $\{a^\uparrow\mid a\in L\}$ is closed under finite intersections, this subbasis is indeed a basis. 
\end{proof}

\begin{lemma} \label{rew}
For $a\in L$, the images of the subsets $a^\uparrow$ and $a^\downarrow$ under $\Gamma:\F(L)\to\F_-$ are given by $\Gamma(a^\uparrow)=\Box_{a^-}$ and $\Gamma(a^\downarrow)=\Diamond_{a^-}$. 
\end{lemma}

\begin{proof}
Let $F$ be a filter of $L$. We have 
$\Gamma(F)\in\Box_{a^-}$ is equivalent to $\Gamma(F)\cap a^-=\varnothing$, which is equivalent to $\Gamma(F)\subseteq a^+$, and thus to $a\in\Psi\Gamma(F)$. Since $\Psi\Gamma$ is the identity, $\Gamma(F)\in\Box_{a^-}$ is therefore equivalent to $a\in F$. So $F\in\Gamma^{-1}(\Box_{a^-})$ iff $F\in a^\uparrow$. The second statement follows because $a^\uparrow$ and $a^\downarrow$ are complements, $\Box_{a^-}$ and $\Diamond_{a^-}$ are complements, and $\Gamma$ is a bijection.
\end{proof}

\begin{lemma} \label{sew}
$\F_-^\Diamond$ has as a subbasis the sets $\Diamond_{a^-}$ where $a\in L$, and $\F_-^\Box$ has as a basis the sets $\Box_{a^-}$ where $a\in L$.
\end{lemma}

\begin{proof}
Since the sets $a^-$ for $a\in L$ are a basis of the topology $\pi^-$, by Lemma~\ref{lem:bases} the topology of $\F_-^\Diamond$ has as a subbasis all sets of the form $\Diamond_{a^-}$ where $a\in L$. This same lemma provides that the topology of $\F_-^\Box$ has as a basis all $\Box_K$ where $K$ is compact and saturated in $\pi^-$. Such $K$ is a closed downset in the $\pi$ topology (see e.g.~\cite[Thm.~6.2]{BBGK10}), so by the dual statement to Proposition~\ref{nmk}, $K=\bigcap\{a^-\mid K\subseteq a^-\}$ and this intersection is down-directed. Any $C\in\F_-$ is closed in $\pi^-$ and hence in $\pi$. Since $X$ is compact under $\pi$, the finite intersection property says that $C$ is disjoint from $K$ iff it is disjoint from some $a^-$ with $K\subseteq a^-$. Thus, $\Box_K=\bigcup\{\Box_{a^-}\mid K\subseteq a^-\}$. This gives that the topology of $\F_-^\Box$ has as a basis all sets of the form $\Box_{a^-}$ where $a\in L$. 
\end{proof}

\begin{theorem}\label{them:F-}
$\Gamma$ is a homeomorphism from $\F(L)$ with the open upset topology to $\F_-^\Box$, a homeomorphism from $\F(L)$ with the open downset topology to $\F_-^{\Diamond}$, and a dual order-homeomorphism from the Priestley space $\F(L)$ to $\F_-^{\Diamond\Box}$.
In particular, $\F_-^\Box$ and $\F_-^\Diamond$ are spectral spaces and $\F_-^{\Diamond\Box}$ is a Priestley space. 
\end{theorem}

\begin{proof}
By Proposition~\ref{nmk}, $\Gamma:\F(L)\to\F_-$ is a bijection. By Proposition~\ref{prop: filter space} and Lemmas~\ref{rew} and~\ref{sew}, it takes a basis for the topology of open upsets of $\F(L)$ to a basis of $\F_-^\Box$, and a subbasis for the topology of open downsets to a subbasis of $\F_-^\Diamond$. Therefore, it takes a subbasis of the Priestley topology on $\F(L)$ to a subbasis of $\F_-^{\Diamond\Box}$. Since $\Gamma$ is a dual order-isomorphism by Proposition~\ref{nmk}, the fact that the order-dual of a Priestley space is again a Priestley space yields that $\F_-^{\Diamond\Box}$ is a Priestley space. Since $\F_-^\Box$ and $\F_-^\Diamond$ are the topologies of open upsets and open downsets of $\F_-^{\Diamond\Box}$, they are spectral spaces.
\end{proof}

We next describe the bounded lattices associated to these spaces. Our primary tool will be a certain free construction which generalizes a similar construction in modal logic \cite{KKV04}. An equivalent construction was considered in \cite[Sec.~4.2]{BKR07}. 

Let $\sf MS$ be the category of unital meet-semilattices and unital meet-semilattice homomorphisms. The forgetful functor ${\sf DL}\to{\sf MS}$ has a left adjoint ${\sf MS}\to{\sf DL}$, which can be described as follows.

\begin{definition} \label{def: Lbox}
Let $M \in \sf MS$ and let $M^\Box$ be the bounded distributive lattice freely generated by the set of symbols $\Box_a$ for $a\in M$ modulo the relations $\Box_1=1$ and $\Box_{a\wedge b}=\Box_a\wedge \Box_b$ for all $a,b\in M$. 
We call $M^\Box$ the {\em bounded distributive lattice freely generated by the unital meet semilattice $M$}. 
\end{definition}

For $M \in \sf MS$ call a downset $S$ of $M$ {\em finitely generated} if $S={\downarrow}F$ for some finite $F\subseteq M$. Let $\mathcal D_{{\rm fg}}(M)$ be the set of finitely generated downsets of $M$ ordered by inclusion. It is easily seen that 
${\downarrow}F\cup{\downarrow}G={\downarrow}(F\cup G)$ and ${\downarrow}F\cap{\downarrow}G={\downarrow}\{a\wedge b : a\in F, b\in G\}$. Thus, $\mathcal D_{{\rm fg}}(M)$ is a bounded sublattice of the powerset of $M$. 
The following appears to be folklore. We include a short proof for the convenience of the reader.

\begin{theorem}\label{thm:Dfg}
For $M \in \sf MS$, $M^\Box$ is isomorphic to $\mathcal D_{{\rm fg}}(M)$.
\end{theorem}

\begin{proof}
The map ${\downarrow}:M\to\mathcal D_{{\rm fg}}(M)$ is a unital meet-semilattice embedding.
Let $f:M\to D$ be a unital meet-semilattice homomorphism to a bounded distributive lattice $D$. It is sufficient to show that there is a unique bounded lattice homomorphism $\overline{f}:\mathcal D_{{\rm fg}}(M)\to D$ such that $\overline{f}({\downarrow}a)=f(a)$ for each $a\in M$.
Let $S\in\mathcal D_{{\rm fg}}(M)$. Then $S={\downarrow}F$ for some finite $F\subseteq S$. 
It follows easily that $\bigvee f(S)=\bigvee f(F)$.
Thus, we may define $\overline{f}:\mathcal D_{{\rm fg}}(M)\to D$ by setting $\overline{f}(S)=\bigvee f(S)$. Clearly $\overline{f}$ preserves finite joins and the bounds. If $S={\downarrow}F$ and $T={\downarrow}G$ for some finite $F,G\subseteq M$, then by distributivity in $D$,
\[
\bigvee f(S)\wedge \bigvee f(T) = \bigvee f(F) \wedge \bigvee f(G) = \bigvee \{ f(a)\wedge f(b) : a\in F,b\in G \}. 
\]
Since $f$ preserves finite meets and $S\cap T={\downarrow}\{a\wedge b : a\in F,b\in G\}$, we have $\overline{f}(S\cap T)=\overline{f}(S)\wedge\overline{f}(T)$. Therefore, $\overline{f}$ is a bounded lattice homomorphism, and clearly $\overline{f}({\downarrow}a)=f(a)$. Since the principal downsets generate $\mathcal D_{{\rm fg}}(M)$ as a lattice, $\overline{f}$ is the unique such bounded lattice homomorphism.
\end{proof}

While the definition of $M^\Box$ is valid for any meet-semilattice $M$, our purpose is to apply it to the meet-semilattice reduct of $L \in \DL$. With this we have the following (cf.~\cite[Thm.~4.10]{BKR07}).

\begin{theorem}\label{thm:dual of Lbox}
The Priestley space $P(L^\Box)$ is order-homeomorphic to $\F(L)$, and thus dually order-homeomorphic to $\F_-^{\Diamond\Box}$. The spectral space $P^+(L^\Box)$ is homeomorphic to $\F_-^{\Box}$, and the spectral space $P^-(L^\Box)$ is homeomorphic to $\F_-^{\Diamond}$.
\end{theorem}

\begin{proof}
By the universal mapping property for $L^\Box$ each unital meet-semilattice homomorphism from $L$ to 2 extends uniquely to a bounded distributive lattice homomorphism from $L^\Box$ to 2. 
Suppose $F$ is a filter of $L^\Box$. Since $\Box$ preserves finite meets, it is easy to see that $\{a\mid \Box_a\in F\}$ is a filter of $L$. Conversely, since filters of $L$ correspond to unital meet-semilattice homomorphisms from $L$ to $2$, the universal mapping property for $L^\Box$ implies that each filter of $L$ is obtained as $\{a\mid \Box_a\in F\}$ for a unique prime filter $F$ of $L^\Box$. This provides a bijective correspondence between $\F(L)$ and the set of prime filters of $L^\Box$. This correspondence takes $a^\uparrow$ to $(\Box_a)^+$ and $a^\downarrow$ to $(\Box_a)^-$. Thus, by Proposition~\ref{prop: filter space} the Priestley space of $L^\Box$ is order-homeomorphic to $\F(L)$. The remaining statements follow from Theorem~\ref{them:F-}.
\end{proof}

We next obtain analogous results for $\F_+$ using the results obtained for $\F_-$ and various symmetries. To begin, we introduce the following notion that is dual to $M^\Box$ (cf.~\cite[Sec.~4.1]{BKR07}). 

\begin{definition} \label{def: Ldiamond}
For a join-semilattice $N$ let $N^\Diamond$ be the bounded distributive lattice freely generated by the set of symbols $\Diamond_a$ for $a\in N$ modulo the relations $\Diamond_0=0$ and $\Diamond_{a\vee b}=\Diamond_a\vee \Diamond_b$ for all $a,b\in N$. 
\end{definition}

Analogous to the situation with $M^\Box$, we have that $N^\Diamond$ is the bounded distributive lattice freely generated by a join-semilattice $N$ with 0. Thus, it provides left adjoint to the forgetful functor from $\sf DL$ to the category of join-semilattices with 0. Using $P^d$ for the order-dual of a poset $P$, it follows from the definitions that if $N$ is a join-semilattice with $0$, then $N^\Diamond$ is isomorphic to the order-dual of $N^{d\,\Box}$, and hence $N^\Diamond$ is isomorphic to $N^{d\,\Box\,d}$. The following is then an immediate consequence of Theorem~\ref{thm:Dfg}:

\begin{corollary}
For a join-semilattice $N$ with $0$, $N^\Diamond$ is isomorphic to the lattice of finitely generated upsets of $N$ ordered by reverse inclusion. 
\end{corollary}

Returning to $L \in \sf DL$, let $P(L^d)=(Y,\mu,\leq)$. Elements of $Y$ are prime filters of $L^d$, hence prime ideals of $L$. For $y\in Y$ the set complement $y^c$ of $y$ in $L$ is a prime filter of $L$. This defines a map $\Delta:Y\to X$ given by $\Delta(y)=y^c$. It is easily seen that $\Delta:P(L^d)\to P(L)$ is a dual order-homeomorphism with $\Delta(a^+)=a^-$ and $\Delta(a^-)=a^+$ for each $a\in L$.

\begin{definition}
Let $\G,\G_-,\G_+$ be the closed sets of $(Y,\mu), (Y,\mu^-)$ and $(Y,\mu^+)$ respectively, ordered by set inclusion. Let $\G^\Diamond,\G^\Box$ and $\G^{\Diamond\Box}$ be the hit, miss, and hit-or-miss topologies on $\G$ arising from the space $(Y,\mu)$; and similarly let $\G_-^\Diamond,\G_-^\Box,\G_-^{\Diamond\Box}$ and $\G_+^\Diamond,\G_+^\Box,\G_+^{\Diamond\Box}$ be those arising form the spaces $(Y,\mu^-)$ and $(Y,\mu^+)$ respectively. 
\end{definition}

Since $\Delta(a^+)=a^-$ and $\Delta(a^-)=a^+$, it is clear that $\Delta$ induces order-isomorphisms between $\G$ and $\F$, between $\G_+$ and $\F_-$, and between $\G_-$ and $\F_+$. The subset $\Box_{a^+}$ of $\G_-$ consists of all closed sets $C$ of $\G_-$ that are disjoint from $a^+$, hence is all $C\in\G_-$ such that $a\not\in y$ for all $y\in C$. Therefore, $\Delta(\Box_{a^+})$ is all $D\in\F_+$ such that $a\in x$ for all $x\in D$, and hence $\Delta(\Box_{a^+})=\Box_{a^-}$. Similar reasoning provides 
\[ \Delta(\Box_{a^+})=\Box_{a^-}, \quad \Delta(\Box_{a^-})=\Box_{a^+}, \quad \Delta(\Diamond_{a^+})=\Diamond_{a^-}, \quad \Delta(\Diamond_{a^-})=\Diamond_{a^+}. \]
Then using $\simeq$ for homeomorphic and $\cong$ for order-homeomorphic, we obtain:

\begin{proposition}\label{flipover}
We have the following homeomorphisms and order-homeomorphisms:
\begin{eqnarray*}
\G_-^\Box\simeq\F_+^\Box, \ \G_-^\Diamond\simeq\F_+^\Diamond \ \mbox{ and } \ \G_-^{\Diamond\Box}\cong\F_+^{\Diamond\Box}; \\ 
\G_+^\Box\simeq\F_-^\Box, \ \G_+^\Diamond\simeq\F_-^\Diamond \ \mbox{ and } \ \G_+^{\Diamond\Box}\cong\F_-^{\Diamond\Box}; \\ 
\G^\Box\simeq\F^\Box, \ \G^\Diamond\simeq\F^\Diamond \ \mbox{ and } \ \G^{\Diamond\Box}\cong\F^{\Diamond\Box}.
\end{eqnarray*}
\end{proposition}

In addition, using $\cong^d$ for dually order-homeomorphic, we have the following consequence of Theorem~\ref{thm:dual of Lbox} and Proposition~\ref{flipover} (cf.~\cite[Thm.~4.6]{BKR07}).  

\begin{corollary}\label{cor:Ldiamond}
The Priestley space $P(L^\Diamond)$ is order-homeomorphic to $\F_+^{\Diamond\Box}$, the spectral space  $P^+(L^\Diamond)$ is homeomorphic to $\F_+^\Diamond$ and the spectral space $P^-(L^\Diamond)$ is homeomorphic to $\F_+^\Box$.
\end{corollary}

\begin{proof}
Theorem~\ref{thm:dual of Lbox} gives the following: 
\[P(L^{\Box}) \cong^d \F_-^{\Diamond\Box}, \quad P^+(L^{\Box})\simeq \F_-^\Box\quad \mbox{and}\quad P^-(L^{\Box})\simeq\F_-^\Diamond.\]
This applied to $L^d$ gives the following:
\[P(L^{d\,\Box}) \cong^d \G_-^{\Diamond\Box}, \quad P^+(L^{d\,\Box})\simeq \G_-^\Box\quad \mbox{and}\quad P^-(L^{d\,\Box})\simeq\G_-^\Diamond.\]
By Proposition~\ref{flipover},
\[P(L^{d\,\Box}) \cong^d \F_+^{\Diamond\Box}, \quad P^+(L^{d\,\Box})\simeq \F_+^\Box\quad \mbox{and}\quad P^-(L^{d\,\Box})\simeq\F_+^\Diamond.\]
Thus, since $L^\Diamond$ is isomorphic to $L^{d\,\Box\,d}$ we have 
\[P(L^\Diamond) \cong \F_+^{\Diamond\Box}, \quad P^-(L^\Diamond)\simeq \F_+^\Box\quad \mbox{and}\quad P^+(L^\Diamond)\simeq\F_+^\Diamond.\]
This provides the stated claims. 
\end{proof}

\section{The hyperspace topologies of a Priestley space} \label{Priestley}

We have shown that $\F_-^{\Diamond\Box}$ and $\F_+^{\Diamond\Box}$ are Priestley spaces and realized them as Priestley spaces of bounded distributive lattices $L^\Box$ and $L^\Diamond$ constructed from $L$. A similar result for $\F^{\Diamond\Box}$ is a simple consequence of this. The essential point is that the order of $(X,\pi,\leq)$ plays no role in the definition of $\F^{\Diamond\Box}$. We can forget the order of $(X,\pi,\leq)$ for this purpose, or better yet, replace it with the partial ordering of equality to obtain the Priestley space $(X,\pi,=)$ that is the Priestley space of the free Boolean extension $B_L$ of $L$. For this Priestley space we have that $\F, \F_-,$ and $\F_+$ coincide. Therefore, Theorem~\ref{thm:dual of Lbox} and Corollary~\ref{cor:Ldiamond} imply the following.

\begin{corollary}\label{cor:free-bool-ext}
$\F^{\Diamond\Box}$ is order-homeomorphic to the Priestley space of $B_L^\Diamond$ and dually order-homeomorphic to that of  $B_L^\Box$. The spectral space of $B_L^\Box$ is homeomorphic to $\F^{\Box}$, and the spectral space of $B_L^\Diamond$ is homeomorphic to $\F^{\Diamond}$.
\end{corollary}

\begin{remark}
We emphasize that neither $B_L^\Diamond$ nor $B_L^\Box$ is a Boolean algebra since the free generation in Definitions~\ref{def: Lbox} and~\ref{def: Ldiamond} happens in $\sf DL$. Therefore, our constructions in Corollary~\ref{cor:free-bool-ext} differ from that in \cite[Prop.~3.12]{KKV04}.
\end{remark}

We next discuss how to lift the order of the Priestley space $P(L) = (X, \pi, \leq)$ to $\F^{\Diamond\Box}$.
A natural candidate is the {\em Egli-Milner order} on the powerset of a poset. 

\begin{definition}  (see, e.g., \cite{Smy91})
For a poset $X$, the Egli-Milner order $\sleq$ on the powerset $\wp(X)$ is given by 
\[A\sqsubseteq B\mbox{ iff }A\subseteq \down B\mbox{ and }B\subseteq \up A.\]
\end{definition}

A note of caution on the name, the Egli-Milner order is not necessarily a partial order on the powerset. It is however a quasi order. We continue to use the terms upset and downset with respect to the Egli-Milner order. As before, a set $U$ is an upset if $F\in U$ and $F\sqsubseteq G$ imply $G\in U$, and a downset is defined similarly. Our interest is in the restriction of the Egli-Milner order to the closed sets $\F$. As mentioned in the introduction, this is something that a number of authors have considered. In particular, the next lemma goes back to Palmigiano \cite[Lem.~26]{Pal04}. For the reader's convenience, we include a short proof. 

\begin{lemma}\label{aqw}
For $a\in L$, the sets $\Box_{a^+},\Diamond_{a^-}$ are downsets in the Egli-Milner order and $\Box_{a^-},\Diamond_{a^+}$ are upsets. Further, if $A,B\in \mathcal{F}$ with $A\not\sqsubseteq B$ then there is $a\in L$ with either
\begin{enumerate}
\item $A\in\Diamond_{a^+}$ and $\,B\in\Box_{a^+}$
\item $A\in\Box_{a^-}$ and $\,B\in\Diamond_{a^-}$. 
\end{enumerate}
\end{lemma}

\begin{proof}
Suppose $A\sqsubseteq B$, so $A\subseteq\down B$ and $B\subseteq\up A$. If $B\in\Box_{a^+}$, then $B\cap a^+=\varnothing$. As $a^+$ is an upset, $\down B\cap a^+=\varnothing$, so $A\cap a^+=\varnothing$, giving $A\in\Box_{a^+}$. Thus $\Box_{a^+}$ is a downset. If $B\in\Diamond_{a^-}$ then $B\cap a^-\neq\varnothing$, so $\up A\cap a^-\neq\varnothing$. As $a^-$ is a downset, $A\cap a^-\neq\varnothing$, giving $A\in\Diamond_{a^-}$. So $\Diamond_{a^-}$ is a downset. The other two cases follow since they are set-theoretic complements of these. 
For the further statement, since $A\not\sqsubseteq B$ either $A\not\subseteq\down B$ or $B\not\subseteq\up A$. Assume the first. Then there is $x\in A$ with $x\not\in\down B$. Therefore, there is a clopen upset $a^+$ with $x\in a^+$ and $\down B\cap a^+=\varnothing$. Thus, $A\in\Diamond_{a^+}$ and $B\in\Box_{a^+}$. The other case gives the second possibility in the statement of the lemma by a similar argument. 
\end{proof}

Thus, $\F^{\Diamond\Box}$ with the quasi-ordering $\sqsubseteq$ satisfies all the conditions to be a Priestley space with the exception of being quasi-ordered rather than partially ordered. Indeed, $\F^{\Diamond\Box}$ is a Stone space since it is the Vietoris space of a Stone space, and Lemma~\ref{aqw} provides that 
incomparable points can be separated by a clopen upset. 
We repair this defect of having a quasi-order rather than a partial order by taking a quotient. This is the approach taken by Palmigiano \cite[Prop.~32]{Pal04} where Proposition~\ref{prop: Palmigiano} was established. We include a short proof for convenience.

\begin{definition}
Let $\sim$ be the equivalence relation on $\F$ associated with the quasi-order $\sqsubseteq$, so $A\sim B$ iff $A\sqsubseteq B$ and $B\sqsubseteq A$. Let $\leq$ be the corresponding partial order on $\F/{\sim}$, and let $q:\F\to \F/{\sim}$ be the quotient map. 
\end{definition}

We equip $\F/{\sim}$ with the quotient topology from $\F^{\Diamond\Box}$ and refer to this as $\F^{\Diamond\Box}/{\sim}$. We recall that a set $S\subseteq \F$ is {\em saturated with respect to} $\sim$ if whenever it contains one member of an equivalence class of $\sim$ it contains all members of that class. This is standard terminology, but differs from another standard usage of saturated introduced in Section~\ref{sec: prelims}. The quotient topology on $\F/{\sim}$ has as opens exactly the sets $q(S)$ where $S$ is open in $\F$ and saturated with respect to $\sim$. 

\begin{proposition} \label{prop: Palmigiano}
With the quotient topology and order, $\F^{\Diamond\Box}/{\sim}$ is a Priestley space with the sets $q(\Box_{a^+})$, $q(\Box_{a^-})$, $q(\Diamond_{a^+})$, $q(\Diamond_{a^-})$ for $a\in L$ as a subbasis. 
\end{proposition}

\begin{proof}
The quotient of a compact space is compact. By Lemma~\ref{aqw} the sets $\Box_{a^+},\Box_{a^-},\Diamond_{a^+},\Diamond_{a^-}$ are upsets or downsets, hence are open sets that are saturated with respect to $\sim$. 
Since this collection of sets is closed under complements, their images are clopen. Finally, if $q(P)\nleq q(Q)$, then $P\not\sqsubseteq Q$, so by Lemma~\ref{aqw} they are separated by sets of the specified form, hence $q(P)$ and $q(Q)$ are separated by a clopen upset and downset. So the quotient is a Priestley space. The indicated sets of the quotient clearly generate a Hausdorff topology coarser than the compact quotient topology, and therefore generate the quotient topology. 
\end{proof}

This quotient has an alternate description. A subset $C$ of a partially ordered set is {\em convex} if $x,y\in C$ and $x\leq z\leq y$ imply $z\in C$. This is equivalent to having $C=\down C\cap \up C$. Each set $A$ has a smallest convex set $A^*$ that contains it given by $A^*=\down A\cap \up A$. This is called its {\em convex hull}. It is not difficult to show that $A\sqsubseteq B$ iff $A^*\sqsubseteq B^*$ and the restriction of the Egli-Milner order to convex sets is a partial order. We next apply this to $\F$. 

\begin{definition}\label{def: C(X)}
Let $\C$ be the set of convex closed subsets of $(X,\pi,\leq)$ under the Egli-Milner order $\sqsubseteq$ with topology generated by taking for all $a\in L$ the subsets $\Box_{a^+},\Box_{a^-},\Diamond_{a^+},\Diamond_{a^-}$ of $\F$ and restricting them to $\C$. We call this the {\em weak hit-or-miss topology} on $\C$. 
\end{definition}

It was pointed out in \cite{BK07} that it is a consequence of \cite{Joh85} that $\C$ with the weak hit-or-miss topology forms a Priestley space. 
A direct proof was given in \cite{VV14}, where $\C(X)$ was denoted by ${\rm V_c}(X)$ and called the {\em Vietoris convex hyperspace}. A bitopological point of view on ${\rm V_c}(X)$ was given in \cite{Lau15}. (see also \cite[Ch.~4]{Jakl2017}). The connection to the quotient construction of \cite{Pal04} was not addressed in the above papers, so we address it next. 

\begin{theorem}\label{prop: Conv}
$\C$ is order homeomorphic to $\F^{\Diamond\Box}/{\sim}$, hence is a Priestley space. 
\end{theorem}

\begin{proof}
Let $\alpha:\F\to\C$ be given by $\alpha(A)=A^*$. Since $\F$ is compact and satisfies the Priestley separation axiom, the upset and downset generated by a closed set is closed (see, e.g., \cite[Prop.~2.3]{BMM02}), so $A^*={\downarrow}A\cap{\uparrow}A$ is closed and convex. Since $A\sqsubseteq B$ iff $A^*\sqsubseteq B^*$, the kernel of $\alpha$ is the equivalence relation $\sim$ and there is an order-isomorphism $\overline{\alpha}:\F/{\sim}\to\C$ with $\overline{\alpha}\circ q = \alpha$. For $a\in L$, we have that $a^+$ is an upset of $X$. So for $A\in\F$ we have $A\cap a^+=\varnothing$ iff $A^*\cap a^+=\varnothing$. It follows that $\alpha(\Box_{a^+})=\Box_{a^+}\cap\C$, hence $\overline{\alpha}q(\Box_{a^+})=\alpha(\Box_{a^+})=\Box_{a^+}\cap\C$. A similar result holds for $\Box_{a^-}$, $\Diamond_{a^+}$ and $\Diamond_{a^-}$. Thus, $\overline{\alpha}$ maps a subbasis of $\F/{\sim}$ to a subbasis of $\C$, hence $\overline{\alpha}$ is an order-homeomorphism. 
\end{proof}

While the convex closed sets $\C$ are realized as a quotient of $\F^{\Diamond\Box}$ under the weak hit-or-miss topology, we can also consider $\C$ as a subspace of $\F^{\Diamond\Box}$. We call this subspace topology the {\em hit-or-miss} topology on $\C$. It has as a subbasis the restrictions of all sets $\Box_K$ and $\Diamond_U$ to $\C$ where $K$ ranges over the compact sets and $U$ over the open sets of $(X,\pi)$. As indicated by the choice of names, the weak hit-or-miss topology is contained in the hit-or-miss topology. One can naturally ask whether these topologies coincide, and whether the hit-or-miss topology of $\C$ is compact. Indeed, this was raised by Palmigiano \cite{Pal04}. The following example shows that this is not the case. 

\begin{example} \label{example}
Suppose $L$ is the bounded distributive lattice whose Priestley space is the poset $X$ shown below, where the clopen sets are the finite sets not containing $x$ and the cofinite sets containing $x$. Therefore, $X$ is the one-point compactification of the discrete space $X \setminus \{x\}$. Thus, $\F$ consists of all subsets that contain $x$ and the finite subsets that do not contain $x$, so the open sets consist of all subsets that do not contain $x$ and the cofinite sets that contain $x$. 
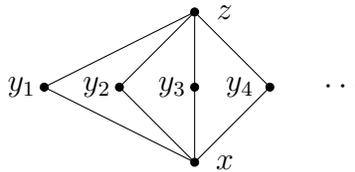
\vspace{2ex}

\begin{center}
\begin{tikzpicture}
\draw[fill] (0,0) circle (1.5pt); 
\draw[fill] (-2,1) circle (1.5pt); \draw[fill] (-1,1) circle (1.5pt); \draw[fill] (0,1) circle (1.5pt); 
\draw[fill] (1,1) circle (1.5pt); \node at (2,1) {$\cdots$}; \draw[fill] (0,2) circle (1.5pt);
\draw (0,0)--(-2,1)--(0,2)--(-1,1)--(0,0)--(0,1)--(0,2)--(1,1)--(0,0);
\node at (.4,0) {$x$}; \node at (.4,2) {$z$}; \node at (-2.3,1) {$y_1$}; \node at (-1.3,1) {$y_2$}; \node at (-.3,1) {$y_3$};
\node at (.6,1) {$y_4$};
\end{tikzpicture}
\end{center}
\vspace{2ex}

We show that $\{x,z\}$ belongs to the closure of $\C$. Since $\{x,z\}$ is not convex, this implies that $\C$ is not closed in $\F^{\Diamond\Box}$. Thus, $\C$ with the hit-or-miss topology is not compact, and hence the hit-or-miss topology is strictly stronger than the weak hit-or-miss topology. For this it is sufficient to show that each basic open neighborhood of $\{x,z\}$ contains a convex set of the form $\{y_n,z\}$, for which it is enough to show that each subbasic open neighborhood $\Box_K\cap\Diamond_U$ of $\{x,z\}$, where $K$ is compact and $U$ is open, contains cofinitely many of the sets of the form $\{y_n,z\}$. 

To see this, let $\{x,z\}\in\Box_K\cap\Diamond_U$. Then $K$ does not contain $x$ and $z$. Since $K$ is a closed set that doesn't contain $x$ and $z$, it must be a finite subset of $\{y_n\mid n\in\mathbb{N}\}$, so cofinitely many $\{y_n,z\}$ are in $\Box_K$. The open set $U$ must contain at least one of $x,z$. If it contains $z$, then all sets $\{y_n,z\}$ belong to $\Diamond_U$. If it does not contain $z$, then it must contain $x$, and as it is open, it must be cofinite. Thus, cofinitely many of the sets $\{y_n,z\}$ are in $\Diamond_U$, and hence also in $\Box_K\cap\Diamond_U$. 
\end{example}

We next turn to the matter of realizing the Priestley space $\C$ as the Priestley space of a bounded distributive lattice constructed from $L$. The following definition originates with Johnstone \cite{Joh82,Joh85} (see also Dunn \cite{Dun95}).

\begin{definition}\label{polp}
Let $L^{\Diamond\Box}$ be the bounded distributive lattice freely generated by the set of symbols $\Diamond_a,\Box_a$ for $a\in L$ modulo the following relations. 
\begin{center}
\begin{tabular}{ll}
$\Diamond_0 = 0$ & $\Diamond_a \vee \Diamond_b = \Diamond_{a \vee b}$ \\
$\Box_1 = 1$ & $\Box_a \wedge \Box_b = \Box_{a \wedge b}$ \\
$\Box_{a \vee b} \le \Box_a \vee \Diamond_b$ \hspace{.5in} & $\Box_a \wedge \Diamond_b \le \Diamond_{a \wedge b}$
\end{tabular}
\end{center}
\end{definition}

This construction can be viewed as follows.

\begin{definition} \label{def: DL diamond box}
Let ${\sf DL}^{\Diamond\Box}$ be the category whose objects are bounded distributive lattices and whose morphisms are pairs $(f,g):L\to K$ such that $f$ preserves finite joins, $g$ preserves finite meets, $g (a \vee b) \le g(a) \vee f(b)$, and $g(a) \wedge f(b) \le f (a \wedge b)$. 
\end{definition}

Phrased in terms of the category ${\sf DL}^{\Diamond\Box}$, Definition~\ref{polp} means that $(\Diamond,\Box):L\to L^{\Diamond\Box}$ is a morphism in ${\sf DL}^{\Diamond\Box}$ and if $(f,g):L\to K$ is a morphism in ${\sf DL}^{\Diamond\Box}$ then there is a unique morphism $h:L^{\Diamond\Box}\to K$ in $\sf DL$ with $(h,h)\circ(\Diamond,\Box) = (f,g)$. 
\begin{center}
\begin{tikzpicture}[scale = 1.2]
\node at (0,1.5) {$L$};
\node at (3,1.5) {$L^{\Diamond\Box}$};
\node at (3,0) {$K$};
\draw[->] (0.35,1.45)--(2.65,1.45);
\draw[->] (0.35,1.55)--(2.65,1.55);
\draw[->] (0.35,1.2) -- (2.65,.1);
\draw[->] (0.3,1.1) -- (2.65,-0.025);
\draw[dashed,->] (3,1.1) -- (3,0.3);
\node at (3.5,.75) {$h$};
\node at (1.7,.85) {$f$};
\node at (1.3,.35) {$g$};
\node at (1.5,1.77) {$\Diamond$};
\node at (1.5,1.25) {$\Box$};
\end{tikzpicture}
\end{center}

This observation together with \cite[p.~81]{Mac98} yields the following.

\begin{proposition}\label{prop: L}
There is a functor $\mathcal R:{\sf DL}\to{\sf DL}^{\Diamond\Box}$ that is the identity on objects and sends a ${\sf DL}$-morphism $f:L\to K$ to the ${\sf DL}^{\Diamond\Box}$-morphism $(f,f)$. This functor has a left adjoint $\mathcal L:{\sf DL}^{\Diamond\Box}\to{\sf DL}$ that sends $L$ to $L^{\Diamond\Box}$.
\end{proposition}

We next give a concrete realization of $L^{\Diamond\Box}$ as a sublattice of the powerset $\wp(B_L)$ of the free Boolean extension $B_L$ of $L$.

\begin{theorem} \label{nop}
$L^{\Diamond\Box}$ is the sublattice of $\wp(B_L)$ generated by the sets $\Box_a={\downarrow}a$ and $\Diamond_b=({\downarrow}b')^c$ where $a,b$ range over all elements of $L$ and $b'$ is the complement of $b$ in $B_L$. 
\end{theorem}

\begin{proof}
Let $S$ be the described sublattice of $\wp(B_L)$ and note that this contains the bounds. We first show that $(\Diamond,\Box):L\to S$ is a morphism in ${\sf DL}^{\Diamond\Box}$ where $\Diamond a=\Diamond_a$ and $\Box a=\Box_a$. Clearly $\Diamond$ preserves finite joins and $\Box$ finite meets. Note that $\Box(a\vee b)\leq \Box a\vee\Diamond b$ is equivalent to ${\downarrow}(a\vee b)\subseteq {\downarrow}a\cup ({\downarrow}b')^c$. This in turn is equivalent to ${\downarrow}(a\vee b)\cap {\downarrow}b'\subseteq {\downarrow}a$, which is equivalent to the obviously true statement ${\downarrow}(a\wedge b')\subseteq{\downarrow}a$. Verifying that $\Box a\wedge \Diamond b\leq \Diamond(a\wedge b)$ is similar. 

It remains to show that this provides the universal solution to this mapping problem. Suppose $D$ is a bounded distributive lattice and $(f,g):L\to D$ is a morphism in  ${\sf DL}^{\Diamond\Box}$. We must show there is a bounded lattice homomorphism $h:S\to D$ with $h\circ\Diamond = f$ and $h\circ\Box = g$. On the generators of $S$ define $h(\Box a) = g(a)$ and $h(\Diamond b) = f(b)$. To show that this extends to a lattice homomorphism, by \cite[p.~86]{BD74} we must show that if a finite meet of generators of $S$ lies below a finite join of generators of $S$, then the meet of their images under $h$ lies below the join of their images under $h$. 

Since the generators $\Box a$ are closed under finite meets, the generators $\Diamond b$ are closed under finite joins, $g$ preserves finite meets, and $f$ preserve finite joins, this amounts to showing that $\Box a\wedge \Diamond b_1\wedge \cdots \wedge \Diamond b_m \,\leq\, \Box a_1\vee \cdots \vee \Box a_n\vee \Diamond b$ 
implies
$g(a)\wedge f(b_1)\wedge \cdots \wedge f(b_m)\,\,\leq\,\, g(a_1)\vee\cdots\vee g(a_n)\vee f(b)$.
The assumption is equivalent to having
\[{\downarrow}a\cap ({\downarrow}b_1'\cup\cdots\cup{\downarrow}b_m')^c\,\,\subseteq\,\, {\downarrow}a_1\cup\cdots\cup{\downarrow}a_n\cup ({\downarrow}b')^c.\]
This is equivalent to 
\[{\downarrow}a\cap{\downarrow}b' \, \subseteq\, {\downarrow}a_1\cup\cdots\cup{\downarrow}a_n\cup {\downarrow}b_1'\cup\cdots\cup{\downarrow}b_m' \]
This implies that $a\wedge b'\leq a_i$ for some $i\leq n$ or $a\wedge b'\leq b_j'$ for some $j\leq m$. Therefore either $a\leq a_i\vee b$ for some $i\leq n$ or $a\wedge b_j\leq b$ for some $j\leq n$. If $a\leq a_i\vee b$, then since $(f,g)$ is a morphism in  ${\sf DL}^{\Diamond\Box}$, we have $g(a)\leq g(a_i\vee b) \leq g(a_i)\vee f(b)$; and if $a\wedge b_j\leq b$, then $g(a)\wedge f(b_j)\leq f(a\wedge b_j)\leq f(b)$. Either case yields the desired conclusion that $g(a)\wedge f(b_1)\wedge \cdots \wedge f(b_m)\,\,\leq\,\, g(a_1)\vee\cdots\vee g(a_n)\vee f(b)$.  
\end{proof}

We now describe the Priestley space of $L^{\Diamond\Box}$ following the suggestion made in \cite[Rem.~4.5.6]{Jakl2017}. As in the previous section, it is convenient to work with an intermediary structure, this time built from the filters and ideals of $L$ (see \cite{Massas2016}). 

\begin{definition}
A {\em weakly prime pair} $(F,I)$ of $L$ consists of a filter $F$ and an ideal $I$ that satisfy 
\begin{align*}
a\vee b \in F &\Rightarrow a\in F\mbox{ or } b\not\in I\\
a\wedge b \in I &\Rightarrow a\not\in F\mbox{ or }b\in I
\end{align*}
Let $\W$ be the set of weakly prime pairs of $L$. Partially order $\W$ by setting $(F,I)\leq (F',I')$ if $F\subseteq F'$ and $I'\subseteq I$ and topologize $\W$ with the following sets, for all $a\in L$, as a subbasis. 
\begin{align*}
\Box_a^\uparrow=\{(F,I)\mid a\in F\}\quad &\quad \Box_a^\downarrow=\{(F,I)\mid a\not\in F\}\\
\Diamond_a^\uparrow\,=\,\{(F,I)\mid a\not\in I\}\quad &\quad \Diamond_a^\downarrow=\{(F,I)\mid a\in I\}
\end{align*}
\end{definition}

\begin{proposition}\label{ghu}
There is an order-homeomorphism $\Lambda:P(L^{\Diamond\Box})\to\W$ given by $\Lambda(G)=(F,I)$ where $F=\{a\mid \Box_a\in G\}$ and $I=\{a\mid \Diamond_a\not\in G\}$. Further, for each $a\in L$ we have $\Lambda(\Box_a^+)=\Box_a^\uparrow$, $\Lambda(\Box_a^-)=\Box_a^\downarrow$, $\Lambda(\Diamond_a^+)=\Diamond_a^\uparrow$ and $\Lambda(\Diamond_a^-)=\Diamond_a^\downarrow$.
\end{proposition}

\begin{proof}
Rather than work directly with prime filters $G$, it is more convenient, and equivalent, to work with their associated bounded lattice homomorphisms $h:L^{\Diamond\Box}\to 2$. Then $\Lambda(G)$ is given by $(F,I)$ where $F=\{a\mid h(\Box_a)=1\}$ and $I=\{a\mid h(\Diamond_a)=0\}$. By the second row of conditions in Definition~\ref{polp} $F$ is a filter, and by the first row $I$ is an ideal. For the further conditions to be weakly prime, let $a\vee b\in F$ and $b\in I$. Then $h(\Box_{a\vee b})=1$ and $h(\Diamond_b)=0$. By the first condition in the third row of conditions $h(\Box_a)=1$, so $a\in F$. The other condition to verify that $(F,I)$ is weakly prime is similar. So $\Lambda(G)$ is a weakly prime pair. 
Since $L^{\Diamond\Box}$ is generated by $\{\Diamond_a,\Box_a\mid a\in L\}$, it follows that any $h:L^{\Diamond\Box}\to 2$ is determined by $(F,I)$, so $\Lambda$ is one-one. Suppose $(F,I)$ is a weakly prime pair. Define $f,g:L\to 2$ by setting $f(a)=0$ iff $a\in I$ and $g(b)=1$ iff $b\in F$. Then $(f,g):L\to 2$ is a morphism in  ${\sf DL}^{\Diamond\Box}$. By Proposition~\ref{prop: L} there is a bounded lattice homomorphism $h:L^{\Diamond\Box}\to 2$ with $h\circ\Diamond = f$ and $h\circ\Box=g$. This homomorphism gives a prime filter $G$ with $\Lambda(G)=(F,I)$. So $\Lambda$ is onto. 

Suppose $G,H$ are prime filters with associated homomorphisms $g,h:L^{\Diamond\Box}\to 2$. Then $G\subseteq H$ iff $g\leq h$ in the usual pointwise order, and since $\{\Diamond_a,\Box_a\mid a\in L\}$ is a generating set of $L^{\Diamond\Box}$, we have $g\leq h$ iff this occurs pointwise on this generating set. Then if $\Lambda(G)=(F,I)$ and $\Lambda(H) = (F',I')$ we have $G\subseteq H$ iff $F\subseteq F'$ and $I'\subseteq I$, hence iff $\Lambda(G)\leq\Lambda(H)$ in the ordering of $\W$. So $\Lambda$ is an order-isomorphism. 

 The definition of $\Lambda$ directly yields $\Lambda(\Box_a^+)=\Box_a^\uparrow$, $\Lambda(\Box_a^-)=\Box_a^\downarrow$, $\Lambda(\Diamond_a^+)=\Diamond_a^\uparrow$ and $\Lambda(\Diamond_a^-)=\Diamond_a^\downarrow$. For instance, if $G$ is a prime filter with $\Lambda(G)=(F,I)$, then $G \in\Box_a^+$ iff $\Box_a\in G$ iff $a\in F$ iff $(F,I)\in\Box_a^\uparrow$. The others are similar. Thus, the bijection $\Lambda$ carries a subbasis of one space to a subbasis of the other, hence is a homeomorphism. 
\end{proof}

For a closed convex set $C$ let $F_C=\{a\mid C\subseteq a^+\}$ and $I_C=\{a\mid C\subseteq a^-\}$. This extends earlier notation used for closed upsets and downsets. It is easy to see that $F_C$ is a filter and $I_C$ is an ideal. For a filter $F$ and ideal $I$ let $C_F=\bigcap\{a^+\mid a\in F\}$ and $C_I=\bigcap\{a^-\mid a\in I\}$. In Proposition~\ref{nmk} we noted that $F\leadsto C_F$ and $C\leadsto F_C$ provide bijections between filters and closed upsets, and $I \leadsto C_I$ and $C \leadsto I_C$ between ideals and closed downsets. For closed convex sets, slightly weaker conditions hold: $F_{C_F}=F$, $I_{C_I}=I$, $C_{F_C}=\up C$, and $C_{I_C}=\down C$. The last two follow from the fact that the upset of a closed set is the intersection of clopen upsets containing it and similarly for the downset.

\begin{proposition} \label{mko}
The maps $\Phi:\W\to\C$ and $\Upsilon:\C\to\W$ given by $\Phi(F,I)=C_F\cap C_I$ and $\Upsilon(C)=(F_C,I_C)$ are mutually inverse order-homeomorphisms. Further, for each $a\in L$ we have $\Phi(\Box_a^\uparrow) = \Box_{a^-}$, $\Phi(\Box_a^\downarrow) = \Diamond_{a^-}$, $\Phi(\Diamond_a^\uparrow) = \Diamond_{a^+}$ and $\Phi(\Diamond_a^\downarrow) = \Box_{a^+}$.
\end{proposition}

\begin{proof}
For $(F,I)\in\W$, surely $C_F\cap C_I$ is a closed convex set, and for $C\in\C$ the pair $(F_C,I_C)$ is a filter and ideal pair. To verify that it is a weakly prime pair, suppose $a\vee b\in F_C$ and $b\in I_C$. Then $C\subseteq (a^+\cup b^+)\cap b^- \subseteq a^+$. So $a\in F$ as required. The other condition is verified similarly. Therefore, both maps are well-defined. Using that $C$ is convex, $C_{F_C}\cap C_{I_C}=\up C\cap \down C = C$, so one composite is the identity. To see that the other composite is the identity, set $C=C_F\cap C_I$. We must show that $F_C=F$ and $I_C=I$. If $a\in F$ then $C\subseteq a^+$, hence $a\in F_C$. For the other containment suppose $a\in F_C$. Then $C\subseteq a^+$, giving $C_F\cap C_I\cap a^-=\varnothing$. Since $C_F$ and $C_I$ are down-directed intersections of clopen sets, by compactness there are $c\in F$ and $d\in I$ with $c^+\cap d^-\cap a^-=\varnothing$. Then $c^+\subseteq d^+\cup a^+$. This gives $c\leq d\vee a$. Since $c\in F$ we then have $d\vee a\in F$. But $d\in I$ and $(F,I)$ is a weakly prime pair, so $a\in F$. The argument to show that $I_C= I$ is similar. Thus, $\Phi$ and $\Upsilon$ are mutually inverse bijections. 

Suppose $(F,I)\in\W$ corresponds to $C\in\C$ under these bijections. Then $C\subseteq a^+$ iff $a\in F$ and $C\subseteq a^-$ iff $a\in I$. Therefore, from the correspondence mentioned above, we have $C_F={\uparrow}C$ and $C_I={\down}C$. If $(F',I')\in\W$ corresponds to $C'\in\C$, it follows that $(F,I)\leq (F',I')$ in $\W$ iff $C\sqsubseteq C'$ in the Egli-Milner order of $\C$. So $\Phi$ and $\Upsilon$ are order-isomorphisms. 

Finally, for $a\in L$ we have $\Phi(\Box_a^\uparrow) = \Box_{a^-}$, $\Phi(\Box_a^\downarrow) = \Diamond_{a^-}$, $\Phi(\Diamond_a^\uparrow) = \Diamond_{a^+}$ and $\Phi(\Diamond_a^\downarrow) = \Box_{a^+}$. To see the first statement, suppose $(F,I)\in\W$ corresponds to $C\in\C$. Then $(F,I)\in\Box_a^\uparrow$ iff $a\in F$ iff $C\subseteq a^+$ iff $C\cap a^-=\varnothing$ iff $C\in\Box_{a^-}$. Similarly, $(F,I)\in\Diamond_a^\uparrow$ iff $a\not\in I$ iff $C\not\subseteq a^-$ iff $C\cap a^+\neq\varnothing$ iff $C\in\Diamond_{a^+}$. The other statements follow as they involve complementary sets and $\Phi$ is a bijection. Since the inverse bijections $\Phi$ and $\Upsilon$ carry one subbasis to another, they are homeomorphisms. 
\end{proof}

Combining Propositions~\ref{ghu} and \ref{mko} yields the following (cf.~\cite[Thm.~4.13]{BKR07}). 

\begin{theorem}
There is a map $\Pi:P(L^{\Diamond\Box})\to\C$ from the Priestley space of $L^{\Diamond\Box}$ to the convex closed sets of the Priestley space of $L$ with the Egli-Milner order and weak hit-or-miss topology given by 
\[ \Pi(G)=\bigcap\{a^+\mid \Box_a\in G\}\cap\bigcap\{a^-\mid\Diamond_a\not\in G\} \]
This map is an order-homeomorphism and for each $a\in L$ we have $\Pi(\Box_a^+)=\Box_{a^-}$, $\Pi(\Box_a^-)=\Diamond_{a^-}$, $\Pi(\Diamond_a^+)=\Diamond_{a^+}$ and $\Pi(\Diamond_a^-)=\Box_{a^+}$.
\label{thm 3.24}
\end{theorem}

By Priestley duality, the order-homeomorphism $\Pi:P(L^{\Diamond\Box})\to\C$ gives rise to a lattice isomorphism $CU(\Pi):CU(\C)\to CU(P(L^{\Diamond\Box}))$ given by $\Pi^{-1}$. Since $\Pi^{-1}(\Box_{a^-}) =\Box_a^+$ and $\Pi^{-1}(\Diamond_{a^+})=\Diamond_a^+$, composing with the natural isomorphism from $CU(P(L^{\Diamond\Box}))\to L^{\Diamond\Box}$ we have a lattice isomorphism $\mu:CU(\C)\to L^{\Diamond\Box}$ with $\mu(\Box_{a^-})=\Box_a$ and $\mu(\Diamond_{a^+})=\Diamond_a$. Taking the inverse of the isomorphism $\mu$ yields the following.

\begin{corollary}\label{mko1}
There is an isomorphism $\varepsilon: L^{\Diamond\Box}\to CU(\C)$ whose behavior on the generating set $\{\Box_a, \Diamond_a\mid a\in L\}$ of $L^{\Diamond\Box}$ is given by $\varepsilon(\Box_a)=\Box_{a^-}$ and $\varepsilon(\Diamond_a)=\Diamond_{a^+}$.  
\end{corollary}

In Sections~\ref{spectral} and~\ref{Priestley} we considered a large number of hyperspaces and their realizations as spectral or Priestley spaces of various distributive lattices. We summarize these results in Table~\ref{table}.

\begin{table}[h!]
  \begin{center}
    \caption{Summary.}
    \label{table}
    \begin{tabular}{c|c|c|c} 
      \textbf{Hyperspace} & \textbf{Realization} & \textbf{Type} & \textbf{Appears}\\
      \hline
      $\F_-^{\Diamond\Box}$ & $P(L^\Box)$ & Priestley & \ref{thm:dual of Lbox} \\
      $\F_-^{\Diamond}$ & $P^-(L^\Box)$ & Spectral & \ref{thm:dual of Lbox} \\
      $\F_-^{\Box}$ & $P^+(L^\Box)$ & Spectral & \ref{thm:dual of Lbox} \\
      \hline
      $\F_+^{\Diamond\Box}$ & $P(L^\Diamond)$ & Priestley & \ref{cor:Ldiamond} \\
      $\F_+^{\Diamond}$ & $P^+(L^\Diamond)$ & Spectral & \ref{cor:Ldiamond} \\
      $\F_+^{\Box}$ & $P^-(L^\Diamond)$ & Spectral & \ref{cor:Ldiamond} \\
      \hline
      $\F^{\Diamond\Box}$ & $P(B_L^{\Diamond}) \cong^d P(B_L^{\Box})$ & Priestley & \ref{cor:free-bool-ext} \\
      $\F^{\Diamond}$ & $P^+(B_L^{\Diamond})$ & Spectral & \ref{cor:free-bool-ext} \\
      $\F^{\Box}$ & $P^+(B_L^{\Box})$ & Spectral & \ref{cor:free-bool-ext} \\
      \hline
      $\C \cong \F^{\Diamond\Box}/{\sim}$ & $P(L^{\Diamond\Box})$ & Priestley & \ref{prop: Conv}, \ref{thm 3.24} \\
          \end{tabular}
  \end{center}
\end{table}

\section{Duality for coalgebras for positive modal logic}

As mentioned in the introduction, there are applications of hyperspace constructions of Priestley spaces to the duality theory for positive modal logic that have been considered by several authors \cite{Pal04,BK07,BKR07,VV14,Lau15,Jakl2017}. In this final section we give a brief account of this, with the focus on the duality of free functor constructions as developed in the previous two sections. 
Recall that in Proposition~\ref{prop: L} we showed that the functor $\mathcal R:{\sf DL}\to{\sf DL}^{\Diamond\Box}$ has a left adjoint $\mathcal L:{\sf DL}^{\Diamond\Box}\to{\sf DL}$ which sends $L$ to $L^{\Diamond\Box}$.

\begin{definition}
Set $\mathcal K=\mathcal L\circ\mathcal R$.
\end{definition}

By Theorem~\ref{prop: Conv}, for a Priestley space $X$, the space $\C(X)$ of convex closed subsets of $X$ with the Egli-Milner order and weak hit-or-miss topology is a Priestley space. Recall that for $A\subseteq X$ we write $A^*$ for the convex hull of $A$. The following appears in \cite[Thm.~13]{VV14} (see also \cite{BK07,Lau15}).

\begin{proposition}
There is an endofunctor $\C:\sf Pries\to Pries$ sending $X$ to $\C(X)$ and $f:X\to Y$ to $\C(f)$ defined by $\C(f)(C)=f(C)^*$ for each $C\in\C(X)$. 
\end{proposition}

Using the Priestley duality functors $P:{\sf DL}\to{\sf Pries}$ and $CU:{\sf Pries}\to{\sf DL}$ we have the situation depicted below.

\begin{center}
\begin{tikzcd}[column sep = 5pc] 
{\sf DL} \arrow[r, shift left = .5ex, "P"] \arrow[d, "\mathcal K"'] & {\sf PS} \arrow[d, "\C"] \arrow[l, shift left, "CU"] \\
{\sf DL}  \arrow[r, shift left = .5ex, "P"] & {\sf PS} \arrow[l, shift left, "CU"]
\end{tikzcd}
\end{center}

The following result shows that the endofunctors $\mathcal K$ and $\C$ are dual to each other (see also \cite[Sec.~4.5.1]{Jakl2017}), where we use $F\simeq G$ to indicate that the functors $F$ and $G$ are naturally isomorphic. 

\begin{theorem}\label{thm:Hwp=wpP}
\begin{enumerate}
\item[]
\item $P\circ\mathcal K\simeq\C\circ P$.
\item $\mathcal K\circ CU\simeq CU\circ\C$.
\end{enumerate}
\end{theorem}

\begin{proof}
We show there is a natural isomorphism $\varepsilon:\mathcal{K}\to CU\circ\C\circ P$. This provides $\mathcal{K}\simeq CU\circ\C\circ P$. Then applying $P$ on the left to each side and using that $P\circ CU\simeq 1_{\sf PS}$ we have $P\circ\mathcal{K}\simeq \C\circ P$; and applying $CU$ on the right to each side of the original equation and using that $P\circ CU\simeq 1_{\sf PS}$ gives $\mathcal{K}\circ CU\simeq CU\circ\C$. 

For each $L\in\DL$ with Priestley space $P(L)=X$, let $\varepsilon_L:L^{\Diamond\Box}\to CU\circ\C(X)$ be the lattice isomorphism of Corollary~\ref{mko1}. For naturality, suppose $M\in\DL$ with Priestley space $P(M)=Y$, and let $f:L\to M$ be a bounded lattice homomorphism. 
\[
\begin{tikzcd}[column sep = 5pc]
L \arrow[d,"f"]&X&L^{\Diamond\Box} \arrow[d, "\mathcal{K}\!f"'] \arrow[r, "\varepsilon_L"] &  CU\,\C(X) \arrow[d, "CU(\C(Pf))"] \\
M&Y \arrow[u,"Pf"]&M^{\Diamond\Box} \arrow[r, "\varepsilon_M"'] & CU\,\C(Y)
\end{tikzcd}
\]
We must show that the square of bounded lattice homomorphisms at right commutes. For this, it is enough to show that it commutes when applied to generators $\Box_a,\Diamond_a$ of $L^{\Diamond\Box}$. We show this for $\Box_a$; the argument for $\Diamond_a$ is similar. By definition, $\varepsilon_L(\Box_a)=\Box_{a^-}$ and then as $\mathcal K\!f(\Box_a)=\Box_{f(a)}$, we have $\varepsilon_M(\mathcal{K}\!f)(\Box_a)=\Box_{f(a)^-}$. Since $\C(Pf) = Pf(\cdot)^*$ we have $CU(\C(Pf))=(Pf(\cdot)^*)^{-1}$. So $CU(\C(Pf))(\Box_{a^-})$ is equal to 
\[
\{D\in\C(Y)\mid Pf(D)^*\in\Box_{a^-}\}.  
\]
But $Pf(D)^*\in\Box_{a^-}\,$ iff  $\,Pf(D)^*\cap a^-=\varnothing\,$ iff $\,Pf(D)^*\subseteq a^+$. Since $a^+$ is convex, we have $Pf(D)^*\subseteq a^+\,$ iff $\,Pf(D)\subseteq a^+\,$. Therefore, the set above is equal to  
\[
\{D\in\C(Y)\mid Pf(D)\subseteq a^+\}. 
\]
But $Pf(D)\subseteq a^+\,$ iff $\,Pf(y)\in a^+$ for each $y\in D$, which occurs iff $a\in Pf(y)$ for each $y\in D$. Since $Pf(y)=f^{-1}(y)$, we have $a\in Pf(y)$ iff $f(a)\in y$ iff $y \in f(a)^+$. Thus, the sets in the displayed equations are equal to 
\[
\{D\in\C(Y)\mid D \subseteq f(a)^+\}.
\]
This set is $\Box_{f(a)^-}$ as required. 
\end{proof}

The next definition is well known (see, e.g., \cite[Def.~5.37]{AHS06}).

\begin{definition}
Let $\sf{C}$ be a category and $\mathcal T:\sf{C} \to \sf{C}$ an endofunctor on $\sf{C}$. An {\em algebra} for $\mathcal T$ is a pair $(A,f)$ where $A$ is an object of $\sf{C}$ and $f:\mathcal T(A) \to A$ is a $\sf{C}$-morphism. A {\em morphism} between algebras $(A_1,f_1)$ and $(A_2,f_2)$ is a $\sf{C}$-morphism $\alpha : A_1 \to A_2$ such that the following square is commutative.
\[
\begin{tikzcd}[column sep = 5pc]
\mathcal T(A_1) \arrow[d, "f_1"'] \arrow[r, "\mathcal T(\alpha)"] &  \mathcal T(A_2) \arrow[d, "f_2"] \\
A_1 \arrow[r, "\alpha"'] & A_2
\end{tikzcd}
\]
Let ${\sf Alg}(\mathcal T)$ be the category whose objects are algebras for $\mathcal T$ and whose morphisms are morphisms of algebras.
\end{definition}

For the next definition see \cite{Dun95}.

\begin{definition}
A {\em positive modal algebra} is a triple $(L,\Diamond,\Box)$ such that $L\in{\sf DL}$ and $\Diamond,\Box$ are unary functions on $L$ satisfying:
\begin{center}
\begin{tabular}{ll}
$\Diamond 0 = 0$ & $\Diamond a \vee \Diamond b = \Diamond (a \vee b)$ \\
$\Box 1 = 1$ & $\Box a \wedge \Box b = \Box (a \wedge b)$ \\
$\Box (a \vee b) \le \Box a \vee \Diamond b$ \hspace{.5in} & $\Box a \wedge \Diamond b \le \Diamond (a \wedge b)$
\end{tabular}
\end{center}
Let $\sf PMA$ be the category of positive modal algebras and bounded lattice homomorphisms preserving $\Diamond$ and $\Box$.
\end{definition}

There is a close connection between the above definition and Definition~\ref{polp}. Indeed, positive modal algebra structures on $L \in \sf DL$ correspond to bounded lattice homomorphisms from $L^{\Diamond\Box}$ to $L$. 
This is reflected 
in the following lemma, which generalizes a similar result for modal algebras (see, e.g., \cite[Prop.~3.12]{KKV04}). The lemma was first established in \cite[Thm.~4.3]{BKR07} as an instance of a more general categorical result. We briefly sketch a direct proof.

\begin{lemma}\label{lem:alg-pma}
$\sf PMA$ is isomorphic to ${\sf Alg}(\mathcal K)$.
\end{lemma}

\begin{proof} 
(Sketch).  
Let $(L,\Diamond,\Box) \in {\sf PMA}$. Since $(\Diamond,\Box) : L \to L$ is a morphism in ${\sf DL}^{\Diamond\Box}$ by Proposition~
\ref{prop: L}, there is a unique $\sf DL$-morphism $\tau_{\Diamond\Box} : L^{\Diamond\Box} \to L$ such that $\tau_{\Diamond\Box}(\Diamond_a) = \Diamond a$ and $\tau_{\Diamond\Box}(\Box_a) = \Box a$ for each $a \in L$. Therefore, $(L, \tau_{\Diamond\Box}) \in {\sf Alg}(\mathcal K)$. Moreover, if $f : L \to K$ is a morphism in $\sf PMA$, then $f$ is also a morphism in ${\sf Alg}(\mathcal K)$.
This gives a covariant functor $\mathcal A : {\sf PMA} \to {\sf Alg}(\mathcal K)$.

Conversely, let $L \in {\sf DL}$ and $\tau : \mathcal K(L) \to L$ be a $\sf DL$-morphism. If we define $\Diamond_\tau$ and $\Box_\tau$ on $L$ by $\Diamond_\tau a = \tau(\Diamond_a)$ and $\Box_\tau a = \tau(\Box_a)$, then $(L, \Diamond_\tau,\Box_\tau) \in {\sf PMA}$. Moreover, if $f : L \to K$ is a morphism in ${\sf Alg}(\mathcal K)$, then 
$f$ is also a $\sf PMA$-morphism. This defines a covariant functor $\mathcal M :  {\sf Alg}(\mathcal K) \to {\sf PMA}$.

Finally, if $(L,\Diamond,\Box) \in {\sf PMA}$, then $\Diamond_{\tau_{\Diamond\Box}} = \Diamond$ and $\Box_{\tau_{\Diamond\Box}} = \Box$; and if $(L, \tau) \in {\sf Alg}(\mathcal K)$, then $\tau_{\Diamond_\tau\Box_\tau}  = \tau$.
Thus, the functors $\mathcal A$ and $\mathcal M$ yield an isomorphism of $\sf PMA$ and ${\sf Alg}(\mathcal K)$.
\end{proof}

The notion of a coalgebra for an endofunctor is dual to that of an algebra (see, e.g., \cite[Def.~2.1]{KKV04}). To describe concretely coalgebras for the endofunctor $\C:{\sf PS}\to{\sf PS}$, we recall that a binary relation $R$ on $X$ is {\em point-closed} if $R[x]$ is closed for each $x\in X$. 
For $U\subseteq X$ let 
\[
\Box_R(U)=\{x\in X\mid R[x]\subseteq U\} \ \mbox{ and } \ \Diamond_R(U)=R^{-1}[U].
\]
In \cite{CJ99} Celani and Jansana generalized J\'onsson-Tarski duality to a duality for positive modal algebras. This resulted in the category of ${\sf K}^+$-spaces and p-morphisms between them. The following definitions are due to them, but we prefer to use modal Priestley space instead of ${\sf K}^+$-space as was done in \cite{BK07}.

\begin{definition} \label{def: modal Priestley space}
A {\em modal Priestley space} is a pair $(X,R)$ where $X$ is a Priestley space and $R$ is a binary relation on $X$ such that
\begin{enumerate}
\item $R$ is point-closed.
\item $U$ clopen upset implies $\Box_R U,\Diamond_R U$ are clopen upsets.
\item $R[x]={\uparrow}R[x]\cap {\downarrow}R[x]$.
\end{enumerate}
\end{definition}

\begin{definition}\label{def:p-mor}
Let $(X_1,R_1)$ and $(X_2,R_2)$ be two modal Priestley spaces. A map $f: X_1\to X_2$ is a {\em p-morphism} provided 
\begin{enumerate}
\item $f$ is order-preserving;
\item $x R_1 z$ implies $f(x) R_2 f(z)$;
\item $f(x) R_2 y$ implies there exist $z,z'\in X_1$ such that $x R_1 z,z'$ and $f(z) \le y \le f(z')$.
\end{enumerate}
Let $\sf MPS$ be the category of modal Priestley spaces and continuous p-morphisms between them.
\end{definition}

It is straightforward to check that $f:X_1 \to X_2$ satisfies Conditions (2) and (3) of Definition~\ref{def:p-mor} iff it satisfies the following condition
\begin{equation}
{\uparrow}fR_1[x] \cap {\downarrow} fR_1[x] = R_2 [f(x)]. \tag{$\dagger$}
\end{equation}

The following lemma is an adaptation of the well-known result in coalgebraic modal logic that the category of descriptive frames is isomorphic to the category of coalgebras for the Vietoris endofunctor on Stone spaces (see, e.g., \cite[Thm.~3.9]{KKV04}). A version of it, based on $\F^{\Diamond\Box}/{\sim}$ instead of $\C$, is given in \cite[Thm.~44]{Pal04}. Another version, using the category of bitopological spaces isomorphic to $\sf PS$, is given in \cite{Lau15}. We sketch a proof in our setting.

\begin{lemma}\label{lem:coalg-mps}
$\sf MPS$ is isomorphic to ${\sf Coalg}(\C)$.
\end{lemma}

\begin{proof}
(Sketch). 
Let $(X,R)$ be a modal Priestley space. Define $\rho_R:X\to\C(X)$ by $\rho_R(x)=R[x]$. It follows from Definition~\ref{def: modal Priestley space} that $\rho_R$ is well defined and continuous. Thus, $(X,\rho_R)$ is a  $\C$-coalgebra. 
If $f:X_1 \to X_2$ is a p-morphism between modal Priestley spaces $(X_1, R_1)$ and $(X_2, R_2)$, then it follows from ($\dagger$) that $f$ is also a morphism between the coalgebras $(X_1,\rho_{R_1})$ and $(X_2, \rho_{R_2})$. 
This defines a covariant functor $\mathcal I : {\sf MPS} \to {\sf Coalg}(\mathcal{P})$. 

Let $(X,\rho)$ be a coalgebra for $\C$. Define $R_\rho$ on $X$ by $R_\rho[x]=\rho(x)$ for each $x\in X$. Since $\rho(x)$ is closed and convex, $(X,R_\rho)$ satisfies Conditions (1) and (3) of Definition~\ref{def: modal Priestley space}. To see Condition (2), for a clopen upset $U$ we have
$\Box_R U = \rho^{-1}(\Box_{U^c})$ and $\Diamond_R U = \rho^{-1}(\Diamond_U)$. This shows that $(X,R_\rho)$ is a modal Priestley space because $\rho$ is continuous and order-preserving and $\Box_{U^c},\Diamond_U$ are clopen upsets of $\C(X)$ (see  Lemma~\ref{aqw}).
If $f$ is a morphism between two coalgebras $(X_1,\rho_1)$ and $(X_2,\rho_2)$, then $f$ is also a p-morphism between the modal Priestley spaces $(X_1,R_{\rho_1})$ and $(X_2, R_{\rho_2})$ since $R_{\rho_1}$ and $R_{\rho_2}$ satisfy $(\dagger)$. 
This defines a covariant functor $\mathcal J : {\sf Coalg}(\C) \to {\sf MPS}$. 

Finally, it is straightforward to see that $R=R_{\rho_R}$ for each $(X,R) \in {\sf MPS}$ and $\rho=\rho_{R_\rho}$ for each $(X,\rho) \in {\sf Coalg}(\C)$. 
Thus, the functors $\mathcal I$ and $\mathcal J$ yield an isomorphism of $\sf MPS$ and ${\sf Coalg}(\C)$.
\end{proof}

As a direct consequence of Theorem~\ref{thm:Hwp=wpP} (see \cite[Thm.~2.5.9]{Jac17} and \cite[Lem.~4.9]{BCM22}), we obtain the following result, which was also obtained in \cite[Sec.~4.3]{BKR07} using the language of spectral spaces. 

\begin{theorem}\label{thm:alg-coalg}
${\sf Alg}(\mathcal K)$ is dually equivalent to ${\sf Coalg}(\C)$.
\end{theorem}

Putting Theorem~\ref{thm:alg-coalg} together with Lemmas~\ref{lem:alg-pma} and~\ref{lem:coalg-mps} yields the following duality theorem of Celani and Jansana \cite{CJ99} (see also Hartonas \cite{Har98}). 

\begin{corollary} 
$\sf PMA$ is dually equivalent to $\sf MPS$.
\end{corollary}

The results 
for the endofunctors $\mathcal K$ and $\C$ involving the $L^{\Diamond\Box}$ construction have obvious analogues for $L^\Diamond$ and $L^\Box$. In particular, they yield coalgebraic proofs of the duality theorems of Goldblatt \cite{Gol89} when his meet- and join-hemimorphisms are unary (see also \cite{CLP91} and \cite{Pet96}). 

Consider the diagram below. The functor $\mathcal K_\Box$ is the composition of the forgetful functor $U:{\sf DL}\to{\sf MS}$ and its left adjoint $L:{\sf MS}\to{\sf DL}$ and $\mathcal K_\Diamond$ is the composition of the forgetful functor $U:{\sf DL}\to{\sf JS}$ and its left adjoint $L:{\sf JS}\to{\sf DL}$.
It is not difficult to see that there are endofunctiors $\mathcal V_+$ and $\mathcal V_-$ on $\sf PS$ that on objects take a Priestley space $(X,\pi,\le)$ to $\mathcal F^{\Diamond\Box}_-$ and $\mathcal F^{\Diamond\Box}_+$ respectively. We follow the standard practice in calling $\mathcal V_+$ the {\em upper Vietoris endofunctor} and $\mathcal V_-$ the {\em lower Vietoris endofunctor}.   
\vspace{1ex}

\begin{center}
\begin{tikzcd}[column sep = 5pc] 
{\sf DL} \arrow[r, shift left = .5ex, "P"] \arrow[d, "\mathcal K_\Box"'] & {\sf PS} \arrow[d, "\mathcal V_+"] \arrow[l, shift left, "CU"] \\
{\sf DL}  \arrow[r, shift left = .5ex, "P"] & {\sf PS} \arrow[l, shift left, "CU"]
\end{tikzcd}
\quad\quad\quad\quad
\begin{tikzcd}[column sep = 5pc] 
{\sf DL} \arrow[r, shift left = .5ex, "P"] \arrow[d, "\mathcal K_\Diamond"'] & {\sf PS} \arrow[d, "\mathcal V_-"] \arrow[l, shift left, "CU"] \\
{\sf DL}  \arrow[r, shift left = .5ex, "P"] & {\sf PS} \arrow[l, shift left, "CU"]
\end{tikzcd}
\end{center}
\vspace{1ex}

A result similar to Theorem~\ref{thm:Hwp=wpP} yields that these diagrams commute up to natural isomorphism. Therefore, the category ${\sf Alg}(\mathcal K_\Box)$ of algebras for $\mathcal K_\Box$ is dually equivalent to the category ${\sf Coalg}({\mathcal V_+})$ of coalgebras for $\mathcal V_+$ and ${\sf Alg}(\mathcal K_\Diamond)$ is dually equivalent to ${\sf Coalg}({\mathcal V_-})$. We next describe these categories of algebras and coalgebras.  

\begin{definition}
A {\em positive $\Box$-algebra} is a pair $(L,\Box)$ such that $L\in{\sf DL}$ and $\Box$ is a unary function on $L$ satisfying $\Box 1 = 1$ and $\Box a \wedge \Box b = \Box (a \wedge b)$. Let $\sf PMA^\Box$ be the category of positive $\Box$-algebras and bounded lattice homomorphisms preserving $\Box$.
\end{definition}

As a consequence of Lemma~\ref{lem:alg-pma} we obtain that $\sf PMA^\Box$ is isomorphic to ${\sf Alg}(\mathcal K_\Box)$.

\begin{definition}
A {\em modal $\Box$-Priestley space} is a pair $(X,R_\Box)$ where $X$ is a Priestley space and $R_\Box$ is a binary relation on $X$ such that $R_\Box(x)$ is a closed upset for each $x\in X$ and $\Box_R U$ is a clopen upset for each clopen upset $U$ of $X$.
\end{definition}

\begin{definition}
Let $(X,R_\Box)$ and $(X',R_\Box')$ be two modal $\Box$-Priestley spaces. A map $f: X\to X'$ is a {\em p-morphism} provided 
\begin{enumerate}
\item $f$ is order-preserving;
\item $xR_\Box z$ implies $f(x)R_\Box'f(z)$;
\item $f(x)R_\Box'y$ implies there is $z\in X$ such that $xR_\Box z$ and $f(z) \le y$.
\end{enumerate}
Let $\sf MPS^\Box$ be the category of modal $\Box$-Priestley spaces and continuous p-morphisms between them.
\end{definition}

An argument similar to Lemma~\ref{lem:coalg-mps} yields that $\sf MPS^\Box$ is isomorphic to ${\sf Coalg}(\mathcal V_+)$ (see also \cite[Sec.~4.2]{BKR07}). This provides a coalgebraic proof of the following duality theorem of Goldblatt \cite{Gol89}: 

\begin{theorem} 
$\sf PMA^\Box$ is dually equivalent to $\sf MPS^\Box$.
\end{theorem}

\begin{definition}
A {\em positive $\Diamond$-algebra} is a pair $(L,\Diamond)$ such that $L\in{\sf DL}$ and $\Diamond$ is a unary function on $L$ satisfying $\Diamond 0 = 0$ and $\Diamond a \vee \Diamond b = \Diamond (a \vee b)$. Let $\sf PMA^\Diamond$ be the category of positive $\Diamond$-algebras and bounded lattice homomorphisms preserving $\Diamond$.
\end{definition}

As a consequence of Lemma~\ref{lem:alg-pma} we obtain that $\sf PMA^\Diamond$ is isomorphic to ${\sf Alg}(\mathcal K_\Diamond)$.
The definition of $\Diamond$-Priestley spaces and the category $\sf MPS_\Diamond$ is similar to the above.

\begin{definition}
A {\em modal $\Diamond$-Priestley space} is a pair $(X,R_\Diamond)$ where $X$ is a Priestley space and $R_\Diamond$ is a binary relation on $X$ such that $R_\Diamond(x)$ is a closed downset for each $x\in X$ and $\Diamond_R U$ is a clopen upset for each clopen upset $U$ of $X$.
\end{definition}

\begin{definition}
Let $(X,R_\Diamond)$ and $(X',R_\Diamond')$ be two modal $\Diamond$-Priestley spaces. A map $f: X\to X'$ is a {\em p-morphism} provided 
\begin{enumerate}
\item $f$ is order-preserving;
\item $xR_\Diamond z$ implies $f(x)R_\Diamond'f(z)$;
\item $f(x)R_\Diamond'y$ implies there exists $z\in X$ such that $xR_\Diamond z$ and $y \le f(z)$.
\end{enumerate}
Let $\sf MPS^\Diamond$ be the category of modal $\Diamond$-Priestley spaces and continuous p-morphisms between them.
\end{definition}

An argument similar to Lemma~\ref{lem:coalg-mps} yields that $\sf MPS^\Diamond$ is isomorphic to ${\sf Coalg}(\mathcal V_-)$ (see also \cite[Sec.~4.1]{BKR07}). This provides a coalgebraic proof of the following duality theorem of Goldblatt \cite{Gol89} (see also \cite{CLP91} and \cite{Pet96}): 

\begin{theorem} 
$\sf PMA^\Diamond$ is dually equivalent to $\sf MPS^\Diamond$.
\end{theorem}

\section*{Acknowledgements}

We are thankful to Tom\'{a}\v{s} Jakl for sharing his knowledge of the subject with us, including pointers to the literature. 
We are also grateful to the referees for constructive comments. 

\bibliographystyle{amsplain}
\bibliography{bibliography}

\end{document}